\documentclass[12pt,a4paper]{article}
\usepackage[utf8x]{inputenc}
\usepackage{amsmath}
\usepackage{amsfonts}
\usepackage{amssymb}
\usepackage{amsthm}
\usepackage{mathrsfs}
\usepackage{fancyhdr}
\usepackage{graphicx}
\usepackage[usenames,x11names]{xcolor}
\usepackage{arabtex}
\usepackage{framed}
\usepackage[top=2cm,bottom=2cm,margin=2cm]{geometry}
\usepackage{cancel}
\usepackage{array}
\usepackage{ulem}

\usepackage{hyperref}

\title{Some applications of the Lagrange inversion formula for the $k$-Fibonacci numbers}
\author{\sc Bakir FARHI \\
Laboratoire de Mathématiques appliquées \\
Faculté des Sciences Exactes \\
Université de Bejaia, 06000 Bejaia, Algeria \\[1mm]
\href{mailto:bakir.farhi@univ-bejaia.dz}{bakir.farhi@univ-bejaia.dz} \\[1mm]
\url{http://farhi.bakir.free.fr/}
}
\date{}

\let\up=\textsuperscript
\let\epsilon=\varepsilon

\def\C{{\mathbb C}}
\def\N{{\mathbb N}}
\def\Z{{\mathbb Z}}
\def\E{{\mathbb E}}

\def\as{\mathrm{As}}

\def\idem{\leavevmode\hbox to 10.6mm{\vrule height .63ex depth -.59ex
    width 10mm\hfill}}

\theoremstyle{plain}
\numberwithin{equation}{section}
\newtheorem{thm}{Theorem}[section]

\newtheorem{prop}[thm]{Proposition}
\newtheorem{coll}[thm]{Corollary}

\newtheorem{thmn}{Theorem}
\newtheorem{colln}[thmn]{Corollary}

\theoremstyle{definition}

\theoremstyle{remark}
\newtheorem{rmk}[thm]{Remark}
\newtheorem{rmks}[thm]{Remarks}

\begin{document}
\maketitle

\begin{abstract}
The aim of this paper consists of providing summation formulas for the $k$-Fibonacci numbers ($k \in \Z$, $k \geq 2$) and their asymptotic equivalents in terms of generalized binomial coefficients. Our main tools are the Lagrange inversion formula and one of its consequences due to Hermite.
\end{abstract}

\noindent\textbf{MSC 2020:} Primary 11B39; Secondary 11B83, 11B65, 11B37. \\
\textbf{Keywords:} $k$-Fibonacci numbers, Fibonacci numbers, linear recurrent sequences, binomial coefficients, the Lagrange inversion formula.

\section{Introduction and Notation}\label{sec1}
Throughout this paper, we let $\N$ denote the set of positive integers and $\N_0 := \N \cup \{0\}$ the set of nonnegative integers. Next, the (compositional) inverse function of a function $f$ is denoted by $f^{<-1>}$ and the coefficient of $z^n$ ($n \in \N_0$) in a power series $\varphi(z)$ is denoted by $[z^n] \varphi(z)$. Further, we let $\binom{a}{b}$ ($a , b \in \Z$) denote the generalized binomial coefficient defined by:
\begin{equation}\label{eq1}
\binom{a}{b} := \begin{cases}
\frac{a (a - 1) \cdots (a - b + 1)}{b!} & \text{if } b \geq 0 \\
\binom{a}{a - b} & \text{if } a \geq b \\
0 & \text{if } a < b < 0
\end{cases} .
\end{equation}
Using this definition, we easily cheek the following formulas:
\begin{align}
\binom{a}{b} & = \binom{a - 1}{b - 1} + \binom{a - 1}{b} ~~~~~~~~~~ (\forall (a , b) \in \Z^2 \setminus \{(0 , 0)\}) , \label{eq2} \\[1mm]
\binom{a}{b} & = \begin{cases}
(-1)^{b - 1} \binom{b - a - 1}{b} & \text{if } b \leq a < 0 \\
(-1)^b \binom{b - a - 1}{b}  & \text{else} 
\end{cases} ~~~~~~~~~~ (\forall (a , b) \in \Z^2) . \label{eq3}
\end{align}
For a given integer $k \geq 2$, the $k$-Fibonacci sequence is the integer sequence ${(F_n^{(k)})}_{n \in \N_0}$ given by the recurrence relation
\begin{equation}\label{eq4}
F_{n + k}^{(k)} = F_n^{(k)} + F_{n + 1}^{(k)} + \dots + F_{n + k - 1}^{(k)} ~~~~~~~~~~ (\forall n \in \N_0) ,
\end{equation}
with the initial values $F_0^{(k)} = F_1^{(k)} = \dots = F_{k - 2}^{(k)} = 0$ and $F_{k - 1}^{(k)} = 1$. The special case $k = 2$ corresponds to the usual Fibonacci sequence, simply denoted by ${(F_n)}_{n \in \N_0}$ and given by
$$
\left\{
\begin{array}{l}
F_0 = 0 , F_1 = 1 , \\
F_{n + 2} = F_n + F_{n + 1} ~~~~ (\forall n \in \N_0)
\end{array}
\right. .
$$
The above linear recurrence relation \eqref{eq4} of order $k$ for the $k$-Fibonacci sequence easily implies the following simpler linear recurrence relation of order $(k + 1)$:
\begin{equation}\label{eq5}
F_{n + k + 1}^{(k)} = 2 F_{n + k}^{(k)} - F_n^{(k)} ~~~~~~~~~~ (\forall n \in \N_0)
\end{equation}
(see \cite{hc} for the details). By relying on \eqref{eq5}, we easily check that we have
\begin{equation}\label{eq17}
\begin{split}
F_n^{(k)} & = 2^{n - k} ~~~~~~~~~~ \text{for } k \leq n < 2 k , \\
F_{2 k}^{(k)} & = 2^k - 1 .
\end{split}
\end{equation}

The $k$-Fibonacci sequence plays an important role in many areas of Mathematics, including Discrete Mathematics and Combinatorics. Here is a simple combinatorial interpretation: given an integer $k \geq 2$, the number of possible ways of writing a positive integer $n$ as a sum of numbers, all belonging to the set $\{1 , 2 , \dots , k\}$ (taking into account the permutations), is exactly $F_{n + k - 1}^{(k)}$. To find a closed form of $F_n^{(k)}$, the theory of linear recurrence sequences associates to the linear recurrence \eqref{eq4} the polynomial equation (known as its \textit{characteristic equation})
\begin{equation}\label{prin-eq1}
x^k = x^{k - 1} + x^{k - 2} + \dots + x + 1 . \tag{$\E_k$}
\end{equation}
Multiplying the both sides of \eqref{prin-eq1} by $(x - 1)$ and rearranging, we obtain the simpler polynomial equation
\begin{equation}\label{prin-eq2}
x^{k + 1} - 2 x^k + 1 = 0 , \tag{$\E_k'$}
\end{equation}
which is nothing else the characteristic equation associated to the linear recurrence \eqref{eq5}. The closed form of $F_n^{(k)}$ then depends on the complex roots of Equation \eqref{prin-eq1} and their multiplicities. Furthermore, in order to establish a simple asymptotic equivalent for $F_n^{(k)}$ (as $n \rightarrow + \infty$), it is necessary to locate the complex roots of Equation \eqref{prin-eq1} in the complex plane, a topic that has been investigated by several authors. Using the famous Rouch\'e theorem of complex analysis, Miles \cite{miles} proved that Equation \eqref{prin-eq1} has exactly $(k - 1)$ pairwise distinct complex roots in the open unit disk $D(0 , 1)$ of the complex plane and a unique complex root $\rho_k$ of modulus greater than $1$; precisely, $\rho_k$ is real and belongs to the interval $(1 , 2)$. In \cite{wol}, Wolfram showed that we have more precisely
\begin{equation}\label{eq6}
2 - \frac{2}{2^k} < \rho_k < 2 .
\end{equation}
For the special case $k = 2$, $\rho_k$ is nothing else the famous \textit{golden ratio}, commonly denoted by $\phi$ ($\phi = \frac{\sqrt{5} + 1}{2} = 1.618\dots$). In \cite{miller}, Miller founded an elementary way to prove the Miles result, which relies on Descartes's rules of signs. Denoting by $\alpha_{k , 1} , \alpha_{k , 2} , \dots , \alpha_{k , k} = \rho_k$ the complex roots of Equation \eqref{prin-eq1}, the theory of linear recurrence sequences shows that $F_n^{(k)}$ has the closed form $F_n^{(k)} = \sum_{\ell = 1}^{k} \lambda_{k , \ell} \alpha_{k , \ell}^n$ ($\forall n \in \N_0$), where the $\lambda_{k , \ell}$'s are complex numbers (independent on $n$) which can be calculated by relying on the initial values of the sequence ${(F_n^{(k)})}_{n}$. In \cite{sj}, Spickerman and Joyner showed that
$$
\lambda_{k , \ell} = \frac{\alpha_{k , \ell}^{k + 1} - \alpha_{k , \ell}^k}{\alpha_{k , \ell} \left[2 \alpha_{k , \ell}^k - (k + 1)\right]} ~~~~~~~~~~ (\forall \ell \in \{1 , 2 , \dots , k\}) ,
$$
and in \cite{dres}, Dresden simplified this to
$$
\lambda_{k , \ell} = \frac{\alpha_{k , \ell} - 1}{\alpha_{k , \ell} \left[2 + (k + 1) (\alpha_{k , \ell} - 2)\right]} = \frac{\alpha_{k , \ell} - 1}{\alpha_{k , \ell} \left[(k + 1) \alpha_{k , \ell} - 2 k\right]} ~~~~~~~~~~ (\forall \ell \in \{1 , 2 , \dots , k\}) .
$$
So we have for all $n \in \N_0$:
$$
F_n^{(k)} = \sum_{\ell = 1}^{k} \frac{\alpha_{k , \ell} - 1}{(k + 1) \alpha_{k , \ell} - 2 k} \,\alpha_{k , \ell}^{n - 1} .
$$
Then, since $\rho_k$ is the largest root of \eqref{prin-eq1} in modulus, we derive the following asymptotic equivalence (as $n \rightarrow + \infty$):
\begin{equation}\label{eq7}
F_n^{(k)} \sim_{+ \infty} \frac{\rho_k - 1}{(k + 1) \rho_k - 2 k} \, \rho_k^{n - 1} .
\end{equation}
Let us denote by $\as_n^{(k)}$ the expression of the right-hand side of \eqref{eq7}; that is
\begin{equation}\label{eq8}
\as_n^{(k)} := \frac{\rho_k - 1}{(k + 1) \rho_k - 2 k} \, \rho_k^{n - 1} 
\end{equation}
(for all integer $k \geq 2$ and all $n \in \N_0$). Next, for all integer $k \geq 2$, define
$$
\epsilon_k := 2 - \rho_k .
$$
From \eqref{eq6}, we have that $\epsilon_k \in (0 , 2^{1 - k}) \subset (0 , 1/2)$. Furthermore, from the fact that $\rho_k$ is a root of Equation \eqref{prin-eq2}, we derive that
\begin{equation}\label{eq9}
\epsilon_k = \frac{1}{(2 - \epsilon_k)^k} ,
\end{equation}
which is a crucial equality for our purpose.

The goal of this paper is first to represent $\rho_k^n$ and $\as_n^{(k)}$ as series involving binomial coefficients. Then, by truncating judiciously the series related to $\as_n^{(k)}$, we derive a representation of $F_n^{(k)}$ as a finite sum involving binomial coefficients. Our results 
are essentially obtained by employing the Lagrange inversion formula and one of its consequences due to Hermite (see \S \ref{sec2}). While some of these results have been previously pointed out by other authors, they were often accompanied by \textit{ad hoc} proofs, and in some instances, contained mistakes (see \S \ref{sec3} for detailed discussion).

\section{The Lagrange inversion formula}\label{sec2}
The Lagrange inversion formula is one of the most important formulas of combinatorics. In its general form, it can be stated as follows

\begin{thmn}[The Lagrange inversion formula]\label{t-lif}
Let $\alpha , z_0 \in \C$ and $f$ be a complex analytic function of $z$ at a neighborhood of $z_0$. Let also $w = w(z)$ be a complex analytic function of $z$ at a neighborhood of $z_0$, which is implicitly defined by{\rm:}
\begin{equation}\label{eq-lag}
w = z_0 + \alpha f(w) .
\end{equation}
Then for every complex analytic function $g$ of $z$ at a neighborhood of $z_0$, we have {\rm(}for $|\alpha|$ sufficiently small{\rm)}
$$
g(w) = g(z_0) + \sum_{n = 1}^{+ \infty} \frac{\alpha^n}{n!} \left[\left(\frac{d}{d z}\right)^{n - 1} \!\!\!\left(f^n g'\right)\right](z_0) .
$$
Taking in particular $g(z) = z$, we get {\rm(}for $|\alpha|$ sufficiently small{\rm)}
$$
w = z_0 + \sum_{n = 1}^{+ \infty} \frac{\alpha^n}{n!} \left[\left(\frac{d}{d z}\right)^{n - 1} \!\!\!\left(f^n\right)\right](z_0) .
$$
\end{thmn}

Note that the last part of this important theorem was proved in 1770 by Lagrange \cite{lag}, who applied it to solve the Kepler equation $y = x + \alpha \sin{y}$, while its generalization (i.e., the first part of Theorem \ref{t-lif}) was obtained by B\"urmann \cite{bur} in 1798. One simple way to prove Theorem \ref{t-lif} uses contour integration in the complex plane, while other more elementary approaches employ algebraic techniques on formal series (see e.g., \cite{com,stan}). From Theorem \ref{t-lif}, we derive the following interesting corollary, which is due to Hermite (see \cite{com}):

\begin{colln}[Hermite]\label{t-her}
Let $\varphi$ be a complex analytic function at a neighborhood of $0$ such that $\varphi(0) \neq 0$ and let $\psi$ be the compositional inverse function of the function $z \rightarrow \frac{z}{\varphi(z)}$ at a neighborhood of $0$ {\rm(}$\psi$ exists because $\left(\frac{z}{\varphi(z)}\right)'(0) = \frac{1}{\varphi(0)} \neq 0${\rm)}. Then for every complex function $g$, which is analytic at a neighborhood of $0$, we have
$$
\frac{z \cdot (g \circ \psi)(z)}{\psi(z) \cdot \left(\left(\frac{z}{\varphi(z)}\right)' \circ \psi(z)\right)} = \sum_{n = 0}^{+ \infty} z^n [z^n]\left(g(z) \varphi(z)^n\right) .
$$ 
\end{colln}

Since this last result of Hermite is unfamiliar, we have preferred to include its proof in this paper.

\begin{proof}[Proof of Corollary \ref{t-her}]
Let f be a primitive of the function $\frac{g}{\varphi}$ at a neighborhood of $0$. By definition of $\psi$, we have $\left(\frac{z}{\varphi(z)} \circ \psi\right)(z) = z$; that is $\psi$ is implicitly given by:
\begin{equation}\label{eq10}
\psi(z) = 0 + z \cdot \left(\varphi \circ \psi\right)(z) .  
\end{equation}
So by applying the first part of Theorem \ref{t-lif} with \eqref{eq10} filling the role of \eqref{eq-lag}; that is with $w$ replaced by $\psi$, $f$ by $\varphi$, $\alpha$ by $z$, $z_0$ by $0$, and $g$ by $f$, we get
\begin{align*}
\left(f \circ \psi\right)(z) & = f(0) + \sum_{n = 1}^{+ \infty} \frac{z^n}{n!} \left[\left(\frac{d}{d z}\right)^{n - 1} \!\!\!\left(\varphi(z)^n f'(z)\right)\right](0) \\
& = f(0) + \sum_{n = 1}^{+ \infty} \frac{z^n}{n} \left[z^{n - 1}\right]\left(g(z) \varphi(z)^{n - 1}\right) ~~~~~~~~~~ (\text{since } f' = \frac{g}{\varphi}) \\
& = f(0) + \sum_{n = 0}^{+ \infty} \frac{z^{n + 1}}{n + 1} \left[z^n\right]\left(g(z) \varphi(z)^n\right) .
\end{align*}
By differentiating with respect to $z$, we get
$$
\psi'(z) \cdot \left(f' \circ \psi\right)(z) = \sum_{n = 0}^{+ \infty} z^n \left[z^n\right]\left(g(z) \varphi(z)^n\right) .
$$
But since
$$
\psi'(z) = \left(\left(\frac{z}{\varphi(z)}\right)^{<-1>}\right)' = \frac{1}{\left(\frac{z}{\varphi(z)}\right)' \circ \psi(z)}
$$
and
\begin{multline*}
\left(f' \circ \psi\right)(z) = \left(\frac{g}{\varphi} \circ \psi\right)(z) = \frac{(g \circ \psi)(z)}{(\varphi \circ \psi)(z)} = \frac{1}{\psi(z)} \cdot \frac{\psi(z)}{(\varphi \circ \psi)(z)} \cdot \left(g \circ \psi\right)(z) \\
= \frac{1}{\psi(z)} \cdot \underbrace{\left(\frac{z}{\varphi(z)} \circ \psi(z)\right)}_{= z} \cdot \left(g \circ \psi\right)(z) = \frac{z}{\psi(z)} \cdot \left(g \circ \psi\right)(z) ,
\end{multline*}
it follows that:
$$
\frac{z \cdot (g \circ \psi)(z)}{\psi(z) \cdot \left(\left(\frac{z}{\varphi(z)}\right)' \circ \psi(z)\right)} = \sum_{n = 0}^{+ \infty} z^n [z^n]\left(g(z) \varphi(z)^n\right) ,
$$
as required.
\end{proof}

\section{The results and the proofs}\label{sec3}
We begin with the following result which express the powers of $\rho_k$ ($k \geq 2$) as real series involving binomial coefficients.

\begin{thm}\label{t1}
For every positive integers $k$ and $n$, with $k \geq 2$, we have
$$
\rho_k^n = 2^n - n 2^{n - k - 1} \sum_{\ell = 0}^{+ \infty} \frac{1}{\ell + 1} \binom{k (\ell + 1) + \ell - n}{\ell} \frac{1}{2^{(k + 1) \ell}} .
$$
In particular, we have
$$
\rho_k = 2 - \frac{1}{2^k} \sum_{\ell = 0}^{+ \infty} \frac{1}{\ell + 1} \binom{k (\ell + 1) + \ell - 1}{\ell} \frac{1}{2^{(k + 1) \ell}} .
$$
\end{thm}

\begin{proof}
Let $k$ and $n$ be two fixed positive integers with $k \geq 2$. By applying the first part of Theorem \ref{t-lif} with \eqref{eq9} filling the role of \eqref{eq-lag}; that is with $w = \epsilon_k$, $f(z) = (2 - z)^{- k}$, $\alpha = 1$, $z_0 = 0$, and $g(z) = (2 - z)^n$, we get
\begin{align*}
\left(2 - \epsilon_k\right)^n & = 2^n + \sum_{\ell = 1}^{+ \infty} \frac{1}{\ell!} \left[\left(\frac{d}{d z}\right)^{\ell - 1} \!\!\!\left((2 - z)^{- k \ell} (- n (2 - z)^{n - 1})\right)\right](0) \\
& = 2^n - n \sum_{\ell = 1}^{+ \infty} \frac{1}{\ell!} \left[\left(\frac{d}{d z}\right)^{\ell - 1}\!\!\!(2 - z)^{- k \ell + n - 1}\right](0) \\
& = 2^n - n \sum_{\ell = 1}^{+ \infty} \frac{1}{\ell} \left[z^{\ell - 1}\right](2 - z)^{- k \ell + n - 1} .
\end{align*}
But according to the generalized binomial formula, we have for all $\ell \in \N$:
$$
\left[z^{\ell - 1}\right](2 - z)^{- k \ell + n - 1} = \binom{- k \ell + n - 1}{\ell - 1} 2^{- k \ell + n - \ell} (-1)^{\ell - 1} = \binom{k \ell - n + \ell - 1}{\ell - 1} \frac{2^n}{2^{(k + 1) \ell}}
$$
(according to \eqref{eq3}). Thus
\begin{align*}
\left(2 - \epsilon_k\right)^n & = 2^n - n \sum_{\ell = 1}^{+ \infty} \frac{1}{\ell} \binom{k \ell - n + \ell - 1}{\ell - 1} \frac{2^n}{2^{(k + 1) \ell}} \\
& = 2^n - n \sum_{\ell = 0}^{+ \infty} \frac{1}{\ell + 1} \binom{k (\ell + 1) + \ell - n}{\ell} \frac{2^n}{2^{(k + 1)(\ell + 1)}} ;
\end{align*}
that is 
$$
\rho_k^n = 2^n - n 2^{n - k - 1} \sum_{\ell = 0}^{+ \infty} \frac{1}{\ell + 1} \binom{k (\ell + 1) + \ell - n}{\ell} \frac{1}{2^{(k + 1) \ell}} ,
$$
as required.
\end{proof}

\begin{rmks}~
\begin{enumerate}
\item In the proof of Theorem \ref{t1}, we have omitted the verification of the convergence of the series $\sum_{\ell = 0}^{+ \infty} \frac{1}{\ell + 1} \binom{k (\ell + 1) + \ell - n}{\ell} \frac{1}{2^{(k + 1) \ell}}$. Note that this convergence can be easily confirmed by using for example the elementary upper bounds $\binom{m}{\ell} \leq \frac{2^m}{\sqrt{m + 2}}$ (for $1 \leq \ell \leq m$).
\item Actually, the formula of $\rho_k$ in the second part of Theorem \ref{t1} has already been discovered by Wolfram \cite{wol}, who obtained it in a more complicated way. 
\end{enumerate}
\end{rmks}

By specializing in Theorem \ref{t1} the integer $k$ to $2$, we derive important formulas for the golden ration $\phi$ and its powers. Precisely, we obtain the following corollary:

\begin{coll}\label{coll1}
For every positive integer $n$, we have
$$
\phi^n = 2^n - n 2^{n - 3} \sum_{\ell = 0}^{+ \infty} \frac{1}{\ell + 1} \binom{3 \ell + 2 - n}{\ell} \frac{1}{2^{3 \ell}} .
$$
In particular, we have
\begin{equation}
\phi = 2 - \sum_{\ell = 0}^{+ \infty} \frac{1}{\ell + 1} \binom{3 \ell + 1}{\ell} \frac{1}{2^{3 \ell + 2}} . \tag*{$\square$}
\end{equation}
\end{coll}

More interestingly, we are going to find representations as series (analogous to those of Theorem \ref{t1}) for the asymptotic equivalent $\as_n^{(k)}$ of $F_n^{(k)}$. To achieve this, we rely on the following key theorem:

\begin{thm}\label{t2}
For every integers $k$ and $a$, with $k \geq 2$, we have
$$
\sum_{\ell = 0}^{+ \infty} \binom{(k + 1) \ell + a}{\ell} 2^{- (k + 1) \ell} = 2^{a + 1} \cdot \frac{\rho_k^{- a}}{(k + 1) \rho_k - 2 k} .
$$
\end{thm}

\begin{proof}
Let $k$ and $a$ be fixed integers, with $k \geq 2$. We apply Corollary \ref{t-her} for the functions $\varphi(z) = (2 - z)^{- k}$ and $g(z) = (2 - z)^{- a - 1}$, which are both analytic at a neighborhood of $0$ and we have $\varphi(0) = 2^{- k} \neq 0$. In this situation, the compositional inverse function $\psi$ of the function $z \mapsto \frac{z}{\varphi(z)} = z (2 - z)^k$ (at a neighborhood of $0$) satisfies (according to \eqref{eq9}): $\psi(1) = \epsilon_k$. So by specializing in formula of Corollary \ref{t-her} the variable $z$ to $1$, we get
$$
\frac{1 \cdot (g \circ \psi)(1)}{\psi(1) \cdot \left(\left(z (2 - z)^k\right)' \circ \psi\right)(1)} = \sum_{\ell = 0}^{+ \infty} 1^{\ell} [z^{\ell}](2 - z)^{- k \ell - a - 1} .
$$
But since $\psi(1) = \epsilon_k$ (as clarified above), $\left(z (2 - z)^k\right)' = (2 - z)^{k - 1} \left(2 - (k + 1) z\right)$, and for all $\ell \in \N_0$: $[z^{\ell}](2 - z)^{- k \ell - a - 1} = \binom{- k \ell - a - 1}{\ell} 2^{- k \ell - a - 1 - \ell} (-1)^{\ell} = \binom{(k + 1) \ell + a}{\ell} 2^{- (k + 1) \ell - a - 1}$ (by using the generalized binomial formula and then \eqref{eq3}), it follows that
$$
\frac{(2 - \epsilon_k)^{- a - 1}}{\epsilon_k (2 - \epsilon_k)^{k - 1} \left(2 - (k + 1) \epsilon_k\right)} = \sum_{\ell = 0}^{+ \infty} \binom{(k + 1) \ell + a}{\ell} 2^{- (k + 1) \ell - a - 1} .
$$
Finally, since $\epsilon_k (2 - \epsilon_k)^k = 1$ (according to \eqref{eq9}), $2 - \epsilon_k = \rho_k$, and $2 - (k + 1) \epsilon_k =$ \linebreak $2 - (k + 1) (2 - \rho_k) = (k + 1) \rho_k - 2 k$, we conclude to
$$
\frac{\rho_k^{- a}}{(k + 1) \rho_k - 2 k} = \sum_{\ell = 0}^{+ \infty} \binom{(k + 1) \ell + a}{\ell} 2^{- (k + 1) \ell - a - 1} ,
$$
which immediately gives the required formula of the theorem.
\end{proof}

\begin{rmk}
The convergence of the series in the formula of Theorem \ref{t2} can be verified by using for example the Stirling formula $n! \sim_{+ \infty} n^n e^{- n} \sqrt{2 \pi n}$.
\end{rmk}

Relying on Theorem \ref{t2}, we are now able to provide a series representation (involving binomial coefficients) for the asymptotic equivalent $\as_n^{(k)}$ of $F_n^{(k)}$. This is the subject of the following theorem:

\begin{thm}\label{t3}
For every nonnegative integers $k$ and $n$, with $k \geq 2$ and $n \neq 1$, we have
$$
\as_n^{(k)} = 2^{n - 2} \sum_{\ell = 0}^{+ \infty} \left[\binom{(k + 1) \ell - n}{\ell} - \binom{(k + 1) \ell - n}{\ell - 1}\right] 2^{- (k + 1) \ell} .
$$
\end{thm}

\begin{proof}
Let $k$ and $n$ be two fixed nonnegative integers such that $k \geq 2$ and $n \neq 1$. By applying Theorem \ref{t2} for $a = - n$, we get
\begin{equation}\label{eq11}
\sum_{\ell = 0}^{+ \infty} \binom{(k + 1) \ell - n}{\ell} 2^{- (k + 1) \ell} = 2^{- n + 1} \cdot \frac{\rho_k^n}{(k + 1) \rho_k - 2 k} .
\end{equation}
Similarly, by applying Theorem \ref{t2} for $a = k + 1 - n$, we get
\begin{equation}\label{eq12}
\sum_{\ell = 0}^{+ \infty} \binom{(k + 1) \ell + k + 1 - n}{\ell} 2^{- (k + 1) \ell} = 2^{k + 2 - n} \cdot \frac{\rho_k^{n - k - 1}}{(k + 1) \rho_k - 2 k} .
\end{equation}
But we remark that
\begin{align*}
\sum_{\ell = 0}^{+ \infty} \binom{(k + 1) \ell + k + 1 - n}{\ell} 2^{- (k + 1) \ell} & = \sum_{\ell = 0}^{+ \infty} \binom{(k + 1)(\ell + 1) - n}{\ell} 2^{- (k + 1) \ell} \\
& = \sum_{\ell = 1}^{+ \infty} \binom{(k + 1) \ell - n}{\ell - 1} 2^{- (k + 1)(\ell - 1)} \\
& = 2^{k + 1} \sum_{\ell = 0}^{+ \infty} \binom{(k + 1) \ell - n}{\ell - 1} 2^{- (k + 1) \ell}
\end{align*}
(since $\binom{(k + 1) \ell - n}{\ell - 1} = 0$ for $\ell = 0$, because $n \neq 1$). Thus Formula \eqref{eq12} is simplified to
\begin{equation}\label{eq13}
\sum_{\ell = 0}^{+ \infty} \binom{(k + 1) \ell - n}{\ell - 1} 2^{- (k + 1) \ell} = 2^{- n + 1} \frac{\rho_k^{n - k - 1}}{(k + 1) \rho_k - 2 k} .
\end{equation}
Next, by subtracting side by side Eq.~\eqref{eq13} from Eq.~\eqref{eq11} and then multiplying by $2^{n - 2}$, we get
\begin{align*}
2^{n - 2} \sum_{\ell = 0}^{+ \infty} \left[\binom{(k + 1) \ell - n}{\ell} - \binom{(k + 1) \ell - n}{\ell - 1}\right] 2^{- (k + 1) \ell} & = \frac{1}{2} \cdot \frac{\rho_k^n - \rho_k^{n - k - 1}}{(k + 1) \rho_k - 2 k} \\
& \hspace*{-3cm} = \frac{1}{2} \cdot \frac{\rho_k - \rho_k^{- k}}{(k + 1) \rho_k - 2 k} \cdot \rho_k^{n - 1} \\
& \hspace*{-3cm} = \frac{1}{2} \cdot \frac{\rho_k - (2 - \rho_k)}{(k + 1) \rho_k - 2 k} \cdot \rho_k^{n - 1} ~~~~~~ (\text{in view of \eqref{eq9}}) \\
& \hspace*{-3cm} = \frac{\rho_k - 1}{(k + 1) \rho_k - 2 k} \rho_k^{n - 1} \\
& \hspace*{-3cm} = \as_n^{(k)} ~~~~~~~~~~~~~~~~~~~~~~~~~~~~~~ (\text{according to \eqref{eq8}}) ,
\end{align*}
as required.
\end{proof}

Observing Formula of Theorem \ref{t3}, and because $As_n^{(k)}$ is close to $F_n^{(k)}$ for $n$ large, a natural question arises: can replacing the series therein with its truncation at some positive integer $N(n , k)$ (depending on $n$ and $k$) yield the $k$-Fibonacci numbers? Working in this direction, we have established the following curious result:

\begin{thm}\label{t4}
For every integers $n , k \geq 2$, we have
$$
F_{n + k - 2}^{(k)} = 2^{n - 2} \sum_{0 \leq \ell \leq \frac{n - 1}{k + 1}} \left[\binom{(k + 1) \ell - n}{\ell} - \binom{(k + 1) \ell - n}{\ell - 1}\right] 2^{- (k + 1) \ell} . 
$$
\end{thm}

To prepare for the proof of Theorem \ref{t4}, let us provisionally define (for every integers $n , k \geq 2$)
$$
G_n^{(k)} := 2^{n - 2} \sum_{0 \leq \ell \leq \frac{n - 1}{k + 1}} \left[\binom{(k + 1) \ell - n}{\ell} - \binom{(k + 1) \ell - n}{\ell - 1}\right] 2^{- (k + 1) \ell} .  
$$
We first establish the two following propositions:

\begin{prop}\label{p2}
Let $n , k \geq 2$ be two integers. Then we have
\begin{equation}\label{eq18}
\begin{split}
G_n^{(k)} & = 2^{n - 2} ~~~~~~~~~~ \text{for } 2 \leq n \leq k + 1 , \\
G_{k + 2}^{(k)} & = 2^k - 1 .
\end{split}
\end{equation}
\end{prop}

\begin{proof}
For $2 \leq n \leq k + 1$, the only integer $\ell$ in the range $[0 , \frac{n - 1}{k + 1}]$ is $\ell = 0$. So the definition of $G_n^{(k)}$ gives
$$
G_n^{(k)} = 2^{n - 2} \left[\binom{- n}{0} - \binom{- n}{- 1}\right] = 2^{n - 2}
$$
(since $\binom{- n}{0} = 1$ and $\binom{- n}{- 1} = 0$), as required.

Next, for $n = k + 2$, the only integers in the range $[0 , \frac{n - 1}{k + 1}] = [0 , 1]$ are $0$ and $1$. So the definition of $G_{k + 2}^{(k)}$ gives
\begin{align*}
G_{k + 2}^{(k)} & = 2^k \left\{\left[\underbrace{\binom{- k - 2}{0}}_{= 1} - \underbrace{\binom{- k - 2}{- 1}}_{= 0}\right] + \left[\underbrace{\binom{-1}{1}}_{= -1} - \underbrace{\binom{-1}{0}}_{= 1}\right] 2^{- k - 1}\right\} \\
& = 2^k \left(1 - 2^{- k}\right) \\
& = 2^k - 1 ,
\end{align*}
as required. This completes the proof of the proposition.
\end{proof}

\begin{prop}\label{p1}
For all integers $n , k \geq 2$, we have
$$
G_{n + k + 1}^{(k)} = 2 G_{n + k}^{(k)} - G_n^{(k)} .
$$
\end{prop}

\begin{proof}
Let $n , k \geq 2$ be two fixed integers. We have by definition
\begin{equation}\label{eq14}
2 G_{n + k}^{(k)} = 2^{n + k - 1} \sum_{0 \leq \ell \leq \frac{n + k - 1}{k + 1}} \left[\binom{(k + 1) \ell - n - k}{\ell} - \binom{(k + 1) \ell - n - k}{\ell - 1}\right] 2^{- (k + 1) \ell} .
\end{equation}
But summing on $\ell$ such that $0 \leq \ell \leq \frac{n + k - 1}{k + 1}$ is equivalent to sum on $\ell$ such that $0 \leq \ell \leq \frac{n + k}{k + 1}$, except perhaps in the case where $\frac{n + k}{k + 1}$ is an integer. Suppose that we are in the last case (i.e., $\frac{n + k}{k + 1} \in \Z$). Then since $\frac{n + k}{k + 1} \geq \frac{k + 2}{k + 1} > 1$, we have $\frac{n + k}{k + 1} \geq 2$. Thus for $\ell = \frac{n + k}{k + 1} \geq 2$, we find that $\binom{(k + 1) \ell - n - k}{\ell} - \binom{(k + 1) \ell - n - k}{\ell - 1} = \binom{0}{\ell} - \binom{0}{\ell - 1} = 0$, implying that the sum for $0 \leq \ell \leq \frac{n + k - 1}{k + 1}$ in \eqref{eq14} can be replaced by the sum for $0 \leq \ell \leq \frac{n + k}{k + 1}$ even in the case where $\frac{n + k}{k + 1} \in \Z$. So we have in all cases
\begin{equation}\label{eq15}
2 G_{n + k}^{(k)} = 2^{n + k - 1} \sum_{0 \leq \ell \leq \frac{n + k}{k + 1}} \left[\binom{(k + 1) \ell - n - k}{\ell} - \binom{(k + 1) \ell - n - k}{\ell - 1}\right] 2^{- (k + 1) \ell} .
\end{equation}
Next, we have by definition
\begin{align*}
G_n^{(k)} & = 2^{n - 2} \sum_{0 \leq \ell \leq \frac{n - 1}{k + 1}} \left[\binom{(k + 1) \ell - n}{\ell} - \binom{(k + 1) \ell - n}{\ell - 1}\right] 2^{- (k + 1) \ell} \\
& = 2^{n - 2} \sum_{1 \leq \ell \leq \frac{n + k}{k + 1}} \left[\binom{(k + 1)(\ell - 1) - n}{\ell - 1} - \binom{(k + 1)(\ell - 1) - n}{\ell - 2}\right] 2^{- (k + 1)(\ell - 1)} \\
& = 2^{n + k - 1} \sum_{1 \leq \ell \leq \frac{n + k}{k + 1}} \left[\binom{(k + 1) \ell - n - k - 1}{\ell - 1} - \binom{(k + 1) \ell - n - k - 1}{\ell - 2}\right] 2^{- (k + 1) \ell} .
\end{align*}
But since $\ell = 0$ gives $\binom{(k + 1) \ell - n - k - 1}{\ell - 1} - \binom{(k + 1) \ell - n - k - 1}{\ell - 2} = \binom{- n - k - 1}{- 1} - \binom{- n - k - 1}{- 2} = 0$ (because $- n - k - 1 \leq - 5$), we have even
\begin{equation}\label{eq16}
G_n^{(k)} = 2^{n + k - 1} \sum_{0 \leq \ell \leq \frac{n + k}{k + 1}} \left[\binom{(k + 1) \ell - n - k - 1}{\ell - 1} - \binom{(k + 1) \ell - n - k - 1}{\ell - 2}\right] 2^{- (k + 1) \ell} . 
\end{equation}
By subtracting (side to side) \eqref{eq16} from \eqref{eq15} and using the formulas  
\begin{align*}
\binom{(k + 1) \ell - n - k}{\ell} - \binom{(k + 1) \ell - n - k - 1}{\ell - 1} = \binom{(k + 1) \ell - n - k - 1}{\ell} \\[-5mm]
\intertext{and} \\[-10mm]
\binom{(k + 1) \ell - n - k}{\ell - 1} - \binom{(k + 1) \ell - n - k - 1}{\ell - 2} = \binom{(k + 1) \ell - n - k - 1}{\ell - 1}
\end{align*}
that result from \eqref{eq2}, we get
\begin{align*}
2 G_{n + k}^{(k)} - G_n^{(k)} & = 2^{n + k - 1} \sum_{0 \leq \ell \leq \frac{n + k}{k + 1}} \left[\binom{(k + 1) \ell - n - k - 1}{\ell} - \binom{(k + 1) \ell - n - k - 1}{\ell - 1}\right] 2^{- (k + 1) \ell} \\
& = G_{n + k + 1}^{(k)} ,
\end{align*}
as required.
\end{proof}

We are now ready to prove Theorem \ref{t4}.

\begin{proof}[Proof of Theorem \ref{t4}]
Let $k \geq 2$ be a fixed integer. We have to show that $F_{n + k - 2}^{(k)} = G_n^{(k)}$ for all integer $n \geq 2$. Since the two sequences ${(F_{n + k - 2}^{(k)})}_{n \geq 2}$ and ${(G_n^{(k)})}_{n \geq 2}$ satisfy the same linear recurrence relation of order $(k + 1)$, which is
$$
u_{n + k + 1} = 2 u_{n + k} - u_n
$$
(according to \eqref{eq5} and to Proposition \ref{p1}) then it suffices to show that $F_{n + k - 2}^{(k)}$ and $G_n^{(k)}$ coincide at the first $(k + 1)$ consecutive values of $n$; that is
\begin{equation}\label{eq19}
F_{n + k - 2}^{(k)} = G_n^{(k)} ~~~~~~~~~~ \text{for } 2 \leq n \leq k + 2 .
\end{equation} 
To confirm \eqref{eq19}, we distinguish the two following cases:\\
\textbullet{} \textit{1\up{st} case}: (if $2 \leq n \leq k + 1$). In this case, we have $k \leq n + k - 2 \leq 2 k - 1$, implying (according to \eqref{eq17}) that $F_{n + k - 2}^{(k)} = 2^{n - 2}$. Since we have also (according to \eqref{eq18}) $G_n^{(k)} = 2^{n - 2}$ then $F_{n + k - 2}^{(k)} = G_n^{(k)}$ for this case. \\
\textbullet{} \textit{2\up{nd} case}: (If $n = k + 2$). In this case, we have $F_{n + k - 2}^{(k)} = F_{2 k}^{(k)} = 2^k - 1$ (according to \eqref{eq17}) and $G_n^{(k)} = G_{k + 2}^{(k)} = 2^k - 1$ (according to \eqref{eq18}). Hence $F_{n + k - 2}^{(k)} = G_n^{(k)}$ also for this case. This confirms \eqref{eq19} and completes this proof.
\end{proof}

By shifting the index $n$ to $(n - k + 2)$ in Theorem \ref{t4}, we derive the following corollary:

\begin{coll}\label{coll2}
For every integers $n$ and $k$ such that $n \geq k \geq 2$, we have
\begin{equation}
F_n^{(k)} = 2^{n - k} \sum_{0 \leq \ell \leq \frac{n - k + 1}{k + 1}} \left[\binom{(k + 1) \ell - n + k - 2}{\ell} - \binom{(k + 1) \ell - n + k - 2}{\ell - 1}\right] 2^{- (k + 1) \ell} . \tag*{$\square$}
\end{equation}
\end{coll}

By taking $k = 2$ in Corollary \ref{coll2}, we derive in particular an expression of the usual Fibonacci numbers as finite sums involving binomial coefficients and nonpositive powers of $8$. This is given by the following corollary:

\begin{coll}\label{coll3}
For all integer $n \geq 2$, we have
\begin{equation}
F_n = 2^{n - 2} \sum_{0 \leq \ell \leq \frac{n - 1}{3}} \left[\binom{3 \ell - n}{\ell} - \binom{3 \ell - n}{\ell - 1}\right] 2^{- 3 \ell} . \tag*{$\square$}
\end{equation}
\end{coll}

We end this paper by deriving from Corollary \ref{coll2} a result which corrects that of Howard and Cooper \cite[Theorem 2.4]{hc}.

\begin{coll}\label{coll4}
For every integers $n$ and $k$ such that $n \geq k \geq 2$ and $n \neq 2 k - 1$, we have
$$
F_n^{(k)} = 2^{n - k} + \sum_{1 \leq \ell \leq \frac{n - k + 1}{k + 1}} (-1)^{\ell} \left[\binom{n - (\ell + 1) k + 2}{\ell} - \binom{n - (\ell + 1) k}{\ell - 2}\right] 2^{n - (k + 1) \ell - k} .
$$
\end{coll}

\begin{proof}
Let $n , k \in \Z$ be fixed with $n \geq k \geq 2$ and $n \neq 2 k - 1$. According to Formula \eqref{eq2}, we have for all $(a , \ell) \in \Z^2 \setminus \{(0 , 0) , (0 , 1)\}$:
$$
\binom{a}{\ell} - \binom{a}{\ell - 1} = \binom{a - 1}{\ell} - \binom{a - 1}{\ell - 2} .
$$
By applying this last formula for each of the couples $(a , \ell) = ((k + 1) \ell - n + k - 2 , \ell)$ ($0 \leq \ell \leq \frac{n - k + 1}{k + 1}$), which are all different from $(0 , 0)$ and $(0 , 1)$, we transform Formula of Corollary \ref{coll2} to
$$
F_n^{(k)} = 2^{n - k} \sum_{0 \leq \ell \leq \frac{n - k + 1}{k + 1}} \left[\binom{(k + 1) \ell + k - n - 3}{\ell} - \binom{(k + 1) \ell + k - n - 3}{\ell - 2}\right] 2^{- (k + 1) \ell} .
$$
Then, by applying Formula \eqref{eq3} for all the binomial coefficients occurring in the last formula, we conclude to
\begin{align*}
F_n^{(k)} & = 2^{n - k} \sum_{0 \leq \ell \leq \frac{n - k + 1}{k + 1}} (-1)^{\ell} \left[\binom{n - (\ell + 1) k + 2}{\ell} - \binom{n - (\ell + 1) k}{\ell - 2}\right] 2^{- (k + 1) \ell} \\
& = 2^{n - k} + \sum_{1 \leq \ell \leq \frac{n - k + 1}{k + 1}} (-1)^{\ell} \left[\binom{n - (\ell + 1) k + 2}{\ell} - \binom{n - (\ell + 1) k}{\ell - 2}\right] 2^{n - (k + 1) \ell - k} ,
\end{align*}  
as required.
\end{proof}

\begin{rmks}~
\begin{enumerate}
\item In \cite[Theorem 2.4]{hc}, Formula of Corollary \ref{coll4} is stated with the erroneous condition of summation $1 \leq \ell \leq \frac{n - 1}{k + 1}$. This condition coincides with ours only for $k = 2$. Therefore, Corollary \ref{coll4} corrects Theorem 2.4 of \cite{hc}.
\item Applying \eqref{eq3} for all the binomial coefficients appearing in Formula of Corollary \ref{coll2}, we transform this later to
$$
F_n^{(k)} = 2^{n - k} + \sum_{1 \leq \ell \leq \frac{n - k + 1}{k + 1}} (-1)^{\ell} \left[\binom{n - (\ell + 1) k + 1}{\ell} + \binom{n - (\ell + 1) k}{\ell - 1}\right] 2^{n - (k + 1) \ell - k} 
$$
(valid for $n \geq k \geq 2$). Note that the last formula and the formula of Corollary \ref{coll4} have the advantage that the binomial coefficients which appear in them are ordinary (i.e., have nonnegative entries). Note also that the two formulas in question are equivalent (we pass from the one to the other via Formula \eqref{eq2}).
\end{enumerate}
\end{rmks}

\rhead{\textcolor{OrangeRed3}{\it References}}

\end{document}